\documentclass[10pt]{amsart}
\usepackage{amsthm,amsmath,amssymb}

\global\overfullrule10pt

\theoremstyle{plain}
\newtheorem{theorem}{Theorem}
\newtheorem{corollary}[theorem]{Corollary}
\newtheorem{lemma}[theorem]{Lemma}
\newtheorem{proposition}[theorem]{Proposition}
\newtheorem*{theorem*}{Theorem}
\newtheorem*{corollary*}{Corollary}
\newtheorem*{proposition*}{Proposition}
\theoremstyle{definition}

\theoremstyle{remark}
\newtheorem{remark}[theorem]{Remark}

\newcommand\CC{{\mathbf C}}
\newcommand\RR{{\mathbf R}}
\newcommand\ZZ{{\mathbf Z}}
\newcommand\NN{{\mathbf N}}

\newcommand\BB{{\mathbf B}}
\newcommand\spr[1]{\langle#1\rangle}
\newcommand\hol{{\text{hol}}}

\renewcommand\Re{\operatorname{Re}}

\newcommand\Bn{{\BB^n}}
\newcommand\bn{{B^n}}
\newcommand\pbn{{\partial\bn}}

\renewcommand\dh{\Delta_h}

\newcommand\FF[5]{{}_#1\!F_#2\Big(\begin{matrix}#3\\#4\end{matrix}\Big|#5\Big)}
\newcommand\FE[4]{F_#1\Big(\begin{matrix}#2\\#3\end{matrix}\Big|#4\Big)}
\newcommand\Hh{$H$-harmonic }
\newcommand\GG[2]{\Gamma\Big(\begin{matrix}#1\\#2\end{matrix}\Big)}
\newcommand\chm{\mathcal H^m}
\newcommand\bhm{\mathbf H^m}
\newcommand\On{{O(n)}}

\newcommand\HH{{\mathcal H}}
\newcommand\dsph{\Delta_{\text{\rm sph}}}
\newcommand\cR{\mathcal R}

\makeatletter
\newcommand{\vast}{\bBigg@{4}}
\newcommand{\Vast}{\bBigg@{5}}
\makeatother

\renewcommand\[{\begin{equation}}
\renewcommand\]{\end{equation}}
\allowdisplaybreaks[4]

\renewcommand\AA{\mathcal A}
\newcommand\jedna{\mathbf 1}
\newcommand\XX{\mathcal X}
\newcommand\YY{\mathcal Y}
\newcommand\nbt{\nabla\!_t\,}
\newcommand\WW{\mathcal W}
\newcommand\bFF[5]{{}_#1\!{\mathbf F}_#2\Big(\begin{matrix}#3\\#4\end{matrix}\Big|#5\Big)}

\begin{document}

\title[Moebius invariant space]{A~Moebius invariant space of $H$-harmonic functions on the ball}
\author[P.~Blaschke]{Petr Blaschke}
\address{Mathematics Institute, Silesian University in Opava,
 Na~Rybn\'\i\v cku~1, 74601~Opava, Czech Republic}
 \email{Petr.Blaschke{@}math.slu.cz}
\author[M.~Engli\v s]{Miroslav Engli\v s}
\address{Mathematics Institute, Silesian University in Opava,
 Na~Rybn\'\i\v cku~1, 74601~Opava, Czech Republic {\rm and }
 Mathematics Institute, \v Zitn\' a 25, 11567~Prague~1,
 Czech Republic}
\email{englis{@}math.cas.cz}
\author{El-Hassan Youssfi}
\address{Aix-Marseille Universit\'e,
Institut de Math\'ematiques de Marseille (I2M) --- UMR 7373,
Site de Saint Charles, 3~place Victor Hugo, Case~19,
13331~Marseille C\'edex~3, France}
\email{el-hassan.youssfi{@}univ-amu.fr}
\thanks{Research supported by GA\v CR grant no.~21-27941S
 and RVO funding for I\v CO~67985840.}
% 32-XX	Several complex variables and analytic spaces
%  32Axx Holomorphic functions of several complex variables
%   32A35 $H^p$-spaces, Nevanlinna spaces of functions in several complex variables
%   32A36 Bergman spaces of functions in several complex variables
% 31-XX Potential theory
%  31Cxx Generalizations of potential theory
%   31C05 Harmonic, subharmonic, superharmonic functions on other spaces
% 35-XX Partial differential equations
%  35Cxx Representations of solutions to partial differential equations
%   35C10 Series solutions to PDEs
% 33-XX Special functions
%  33Cxx Hypergeometric functions
%   33C55 Spherical harmonics
%   33C65 Appell, Horn and Lauricella functions
%   33C70 Other hypergeometric functions and integrals in several variables
\begin{abstract} We~describe a Dirichlet-type space of $H$-harmonic functions,
i.e. functions annihilated by the hyperbolic Laplacian on~the unit ball of the real $n$-space,
as~the analytic continuation (in~the spirit of Rossi and Vergne) of the corresponding
weighted Bergman spaces. Characterizations in terms of derivatives are given,
and the associated semi-inner product is shown to be Moebius invariant.
We~also give a formula for the corresponding reproducing kernel.
Our~results solve an open problem addressed by M.~Stoll in his book 
``Harmonic and subharmonic function theory on the hyperbolic ball''
(Cambridge University Press, 2016).
\end{abstract}

\maketitle

\section{Introduction}
Let $\bn$ be the unit ball in the real $n$-space~$\RR^n$, $n>2$, equipped with the usual hyperbolic
metric $2(1-|x|^2)^{-1}\,|dx|$. The~associated hyperbolic Laplace operator
$$ \dh f(x) := (1-|x|^2) [(1-|x|^2)\Delta f(x) + 2(n-2) \spr{x,\nabla f(x)}]  $$
is invariant under the Moebius transformations of~$\bn$.
Functions on $\bn$ annihilated
by $\dh$ are called hyperbolic-harmonic, or \Hh for short.
The~\emph{weighted \Hh Bergman space}
$$ \HH_s(\bn) := \{ f\in L^2(\bn,d\rho_s): \;f\text{ is \Hh on }\bn\}  $$
consists of all \Hh functions on $\bn$ square-integrable with respect to the measure
$$ d\rho_s(x) := \frac{\Gamma(\frac n2+s+1)}{\pi^{n/2}\Gamma(s+1)} (1-|x|^2)^s\,dx,  \qquad s>-1, $$
where $dx$ denotes the Lebesgue volume on~$\RR^n$. The~restriction on $s$ ensures that these spaces are nontrivial,
and the factor $\frac{\Gamma(\frac n2+s+1)}{\pi^{n/2}\Gamma(s+1)}$ makes $d\rho_s$ a probability measure,
so~that $\|\jedna\|=1$. By~Green's formula, \Hh functions possess the \emph{invariant mean-value property}:
namely, if $\dh f=0$ and $z\in\bn$, then $f(z)$ equals the mean value, with respect to the Moebius-invariant
measure $d\tau(x)=(1-|x|^2)^{-n}\,dx$, over any Moebius ball in $\bn$ centered at~$x$.
It~follows by a standard argument that the point evaluations $f\mapsto f(x)$ at any $x\in\bn$
are continuous linear functionals on each~$\HH_s$, $s>-1$, and therefore there exists a reproducing
kernel for $\HH_s$ (the~\emph{weighted \Hh Bergman kernel}), namely a function $K_s(x,y)$ on $\bn\times\bn$,
\Hh in both variables and such that
$$ f(x) = \int_\bn f(y) K_s(x,y) \, d\rho_s(y) \qquad \forall x\in\bn, \forall f\in\HH_s.  $$

For the analogous weighted Bergman spaces of \emph{holomorphic}, rather than $H$-harmonic, functions on the
unit ball $\Bn\cong B^{2n}$ of $\CC^n\cong\RR^{2n}$,
$$ \AA_s := \{f\in L^2(\Bn,d\rho_s): \; f\text{ is holomorphic on }\Bn \} , $$
the reproducing kernels are given by the simple formula
$$ K_s^\hol(x,y) = (1-\spr{x,y})^{-n-1} .  $$
It~is a remarkable fact --- which prevails in the much more general context of bounded
symmetric domains, constituting the ``analytic continuation'' of the principal series
representations of certain semisimple Lie groups, cf.~Rossi and Vergne~\cite{RV} ---
that these weighted Bergman kernels $K^\hol_s(x,y)$, $s>-1$, continue to be positive definite
kernels in the sense of Aronszajn~\cite{Aro} for all $s\ge-n-1$, yielding thus an
``analytic continuation'' of the spaces~$\AA_s$. (One~calls the interval $-n-1\le s<+\infty$
the \emph{Wallach set} of~$\Bn$.) For $s>-n-1$, these spaces $\AA_s$ turn out to be Besov-type
spaces of analytic functions; for $s=-n-1$, the kernel $K_{-n-1}^\hol(x,y)$ becomes constant~one,
and the corresponding reproducing kernel Hilbert space thus reduces just to the constants.
However, a~much more interesting space arises as the ``residue'' of $\AA_s$ at $s=-n-1$:
namely, the~limit
$$ \lim_{s\searrow-n-1} \frac{K^\hol_s(x,y)-1}{s+n+1} = \log\frac1{1-\spr{x,y}} $$
is~a~positive definite kernel on $\Bn\times\Bn$, and the associated reproducing kernel
Hilbert space is nothing else but the familiar \emph{Dirichlet space} on~$\Bn$.
Furthermore, the semi-inner product in this space turns out to be \emph{Moebius invariant},
in~the sense that
$$ \spr{f,g} = \spr{f\circ\phi,g\circ\phi}  $$
for any biholomorphic self-map $\phi$ of~$\Bn$. See e.g.~Chapter~6.4 in~Zhu~\cite{Zhu}.

It~has recently been shown by two of the current authors that a similar situation as described
in the previous paragraph prevails also for spaces of $M$-harmonic functions on~$\Bn$
(i.e.~functions annihilated by the Poincar\'e Laplacian on~$\Bn$ invariant under the action
of the biholomorphic self-maps of~$\Bn$), as~well as for spaces of harmonic and pluriharmonic functions.
The~aim of the current paper is to treat the case of \Hh functions on~$\bn$.

Quite surprisingly (at~least for the current authors), it~turns out that the situation in the
\Hh case is strikingly different from all the cases above, in~that the ``Wallach set'' of those
real $s$ for which $K_s(x,y)$ continue to be positive definite kernels, does not have the form
of an interval, with an appropriate ``Dirichlet space'' arising as the limit (or,~rather, residue)
at~the lower endpoint. Instead, $K_s$~stop to be positive definite some way below $s=-1$,
only to resurface into positive definite kernels at the point $s=-n$.
In~particular, the ``\Hh Wallach set'' of~$\bn$ is thus disconnected.
% (So~$s=-n$~is, in~a~sense, a~discrete point of ).
Furthermore, although we do not have
any closed formula for the corresponding reproducing kernel, we~are able to show that the
associated semi-inner product~is, miraculously, indeed invariant under composition with
Moebius maps of~$\bn$. In~particular, this gives an answer to Question~2 on page~180 of the
book of M.~Stoll~\cite{St}, in~exhibiting a Moebius-invriant Hilbert space of \Hh functions.

Finally, just as the ordinary (i.e.~holomorphic) Dirichlet space on the disc consists of all holomorphic
functions whose first derivative is square-integrable over the disc, we~give characterizations of our
\Hh Dirichlet space in terms of square-integrability of derivatives of appropriate order.
This is very similar in spirit to the results of Grellier and Jaming \cite{Ja}~\cite{GJ}.

The~details of the construction of our ``\Hh Dirichlet space'' are presented in Section~\ref{sec3},
after recalling the necessary background material in Section~\ref{sec2}. The~characterization in terms
of derivatives is given in Section~\ref{sec4}, and the invariance of the semi-inner product is proved
in Section~\ref{sec5}. In~Section~\ref{sec6}, we~identify the reproducing kernel of our Dirichlet space.
The~last section, Section~\ref{sec7}, contains some final remarks and comments.

Throughout the paper, the notation
$$ A\asymp B $$
means that
$$ cA \le B \le \frac1c A  $$
for some $0<c\le1$ independent of the variables in question. If~only the first of these inequalities
is fulfilled, we~write $A\lesssim B$; and $A\sim B$ means that $A/B$ tends to~1.
The~partial derivatives $\partial/{\partial x_j}$ are commonly abbreviated just to~$\partial_j$,
and similarly $\partial_r$ stands for $\partial/\partial r$ and so forth;
and for $\alpha=(\alpha_1,\dots,\alpha_n)\in\NN^n$ a multi-index, $\partial^\alpha$ stands for
$\partial_1^{\alpha_1}\dots\partial_n^{\alpha_n}$.
To~make typesetting a little neater, the shorthand
$$ \GG{a_1,a_2,\dots,a_k}{b_1,b_2,\dots,b_m} := \frac{\Gamma(a_1)\Gamma(a_2)\dots\Gamma(a_k)}
   {\Gamma(b_1)\Gamma(b_2)\dots\Gamma(b_m)}  $$
is often employed throughout the paper.
Finally, $\ZZ$, $\NN$, $\RR$ and $\CC$ denote the sets of all integers, all nonnegative integers,
all real and all complex numbers, respectively.

\section{Notation and preliminaries} \label{sec2}
Let $\bn$ be the unit ball in~$\RR^n$, $n>2$. The~orthogonal transformations
$$ x \longmapsto Ux, \qquad x\in\RR^n, \; U\in\On,  $$
map both $\bn$ and its boundary $\pbn$ (the unit sphere) onto themselves,
and so do the Moebius transformations
$$ \phi_a(x) := \frac{a|x-a|^2+(1-|a|^2)(a-x)}{[x,a]^2}  $$
interchanging the origin $0\in\RR^n$ with some point $a\in\bn$; here $[x,a]$ is defined~as
$$ [x,a]:=\sqrt{1-2\spr{x,a}+|x|^2|a|^2}.  $$
The~group generated by the $\phi_a$, $a\in\bn$, and $\On$ via composition is called the Moebius group of~$\bn$;
it~is actually true that any element of this group can be written just as $U\phi_a$ with some
$U\in\On$ and $a\in\bn$. The~hyperbolic Laplacian
$$ \dh f(x) := (1-|x|^2) [(1-|x|^2)\Delta f(x) + 2(n-2) \spr{x,\nabla f(x)}]  $$
is Moebius invariant, i.e.~commutes with the action of this group. Functions on $\bn$ annihilated
by $\dh$ are called hyperbolic-harmonic, or \Hh for short.
The~\emph{weighted \Hh Bergman space}
$$ \HH_s(\bn) := \{ f\in L^2(\bn,d\rho_s): \;f\text{ is \Hh on }\bn\}  $$
consists of all \Hh functions on $\bn$ square-integrable with respect to the measure
\[ d\rho_s(x) := \frac{\Gamma(\frac n2+s+1)}{\pi^{n/2}\Gamma(s+1)} (1-|x|^2)^s\,dx,  \qquad s>-1, \label{WD} \]
where $dx$ denotes the Lebesgue volume on~$\RR^n$. The~restriction on $s$ ensures that these spaces are nontrivial,
and the factor $\frac{\Gamma(\frac n2+s+1)}{\pi^{n/2}\Gamma(s+1)}$ makes $d\rho_s$ a probability measure,
so~that $\|\jedna\|=1$.

For~an integer $m\ge0$, let $\chm$ be the space of restrictions to the unit sphere $\pbn$
of harmonic polynomials on~$\RR^n$ homogeneous of degree~$m$. We~refer to \cite{ABR},
especially Chapter~5, for the Peter-Weyl decomposition
\[ L^2(\pbn,d\sigma) = \bigoplus_{m=0}^\infty \chm   \label{WA}  \]
under the action of the orthogonal group $O(n)$ of rotations of~$\RR^n$.
Here and throughout $d\sigma$ stands for the normalized surface measure on the unit sphere~$\pbn$.
Performing such a decomposition on each sphere $|z|\equiv\text{const.}$ leads to the analogous
Peter-Weyl decomposition
$$ \HH_s = \bigoplus_m \bhm,  $$
of the weighted \Hh Bergman spaces, where $\bhm$ is the space of ``solid harmonics''
$$ \bhm = \{f\in C(\overline{\bn}): \;f\text{ is \Hh on }\bn \text{ and }f|_\pbn\in\chm\} ,  $$
and the norm of $f=\sum_m f_m$, $f_m\in\bhm$, is~given~by
\[ \|f\|^2_s = \sum_{m=0}^\infty I_m(s) \|f_m\|^2_\pbn,  \label{TJ}  \]
with the coefficients $I_m(s)$ given by the explicit formula
\[ I_m(s) := \frac{\Gamma(\frac n2+s+1)}{\Gamma(\frac n2)\Gamma(s+1)} \int_0^1 t^{m+\frac n2-1} (1-t)^s S_m(t)^2 \,dt , \label{YE} \]
where
\[ S_m(t) := \frac{(n-1)_m}{(\frac n2)_m} \FF21{m,1-\tfrac n2}{m+\frac n2}t  \]
with the Gauss hypergeometric function $_2F_1$ and the Pochhammer symbol $(x)_m:=x(x+1)\dots(x+m-1)$.
Also, any $f_m\in\bhm$ is necessarily of the form
\[ f_m(r\zeta) = S_m(r^2) r^m f_m(\zeta), \quad f_m|\pbn\in\chm, \qquad 0\le r\le1, \zeta\in\pbn.  \label{YC} \]
It~follows that the space $\HH_s$ has reproducing kernel given~by
\[ K_s(x,y) = \sum_{m=0}^\infty \frac{S_m(|x|^2) S_m(|y|^2)}{I_m(s)} Z_m(x,y),  \]
where $Z_m(x,y)$, the zonal harmonic of degree~$m$, is~the reproducing kernel of~$\chm$
extended by homogeneity to all $x,y\in\bn$; see Chapter~8 in~\cite{ABR}.

The~reader is referred to \cite{St} for a detailed exposition of the facts above,
with convenient overview and further developments in \cite{St2} and~\cite{Ur}.
Note however that, for ease of notation, our $S_m(t)$ is $S_m(\sqrt t)$ in the notation
of these references.

\section{Analytic continuation} \label{sec3}
Throughout this paper, we~assume that $n>2$. (For $n=2$, \Hh~functions coincide with ordinary harmonic ones,
and most things work out differently and, in~fact, become much simpler; cf.~Section~6 in~\cite{EY2}.)

\begin{proposition} \label{PA}
The coefficient functions $I_m(s)$, $m>0$, extend to meromorphic functions of $s$ on the entire complex plane,
with a simple pole at $s=-n$, with residue $(n-1)_m/\Gamma(m)$ for $n>2$ even, and $2(n-1)_m/\Gamma(m)$ for $n>2$~odd.
\end{proposition}

For~the proof, we will need the following simple lemma.

\begin{lemma}\label{Smlemma}
Let $m\in\NN$, $n>1$, $s>-1$. Then for all $\alpha\geq 0$:
\[ \label{Smlemmaform}
\int_0^1 t^{m+\frac n2-1}(1-t)^{s+\alpha}S_m(t)\,dt=
\frac{(n-1)_m\Gamma(s+1+\alpha)\Gamma(n+s+\alpha)\Gamma(\frac n2)}{\Gamma(m+n+s+\alpha)\Gamma(\frac n2+s+1+\alpha)}.  \]
\end{lemma}

\begin{proof}
From definition of $S_m$ we have
\begin{align*}
\int_0^1 t^{m+\frac n2-1}(1-t)^{s+\alpha}S_m(t)\,dt
&= \frac{(n-1)_m}{(\frac n2)_m} \int_0^1 t^{m+\frac n2-1}(1-t)^{s+\alpha} \FF21{m,1-\frac n2}{m+\frac n2}t \,dt\\
&= \frac{(n-1)_m}{(\frac n2)_m} \int_0^1 t^{m+\frac n2-1}(1-t)^{s+\alpha}
\sum_{k=0}^\infty \frac{(m)_k(1-\frac n2)_k}{(m+\frac n2)_k k!} t^k \,dt.
\end{align*}
The series converges uniformly for all $t\in [0,1]$ when $n>1$ by a simple ratio test.
Therefore we can integrate term by term to get
\begin{align*}
&= \frac{(n-1)_m}{(\frac n2)_m} \sum_{k=0}^\infty \frac{(m)_k(1-\frac n2)_k}{(m+\frac n2)_k k!}
 \frac{\Gamma(m+\frac n2+k)\Gamma(s+\alpha+1)}{\Gamma(m+\frac n2+s+\alpha+1+k)} \\
&= \frac{(n-1)_m\Gamma(m+\frac n2)\Gamma(s+\alpha+1)}{(\frac n2)_m\Gamma(m+\frac n2+s+\alpha+1)} \FF21{m,1-\frac n2}{m+\frac n2+s+\alpha+1}1 \\
&= \frac{(n-1)_m\Gamma(s+1+\alpha)\Gamma(n+s+\alpha)\Gamma(\frac n2)}{\Gamma(m+n+s+\alpha)\Gamma(\frac n2+s+1+\alpha)},
\end{align*}
by Gauss's summation formula \cite[\S2.1~(14)]{BE}.  \end{proof}

\begin{proof}[Proof of Proposition~\ref{PA}]
Using the formula \cite[\S2.10~(12)]{BE} on $S_m(t)$ we obtain
\begin{align*}
S_m(t) &= \sum_{k=0}^{n-2}\frac{(m)_k(1-\frac n2)_k}{k!(2-n)_k}(1-t)^k \\
& \qquad + \frac{(-1)^n}{\Gamma(m)\Gamma(1-\frac n2)\Gamma(n-1)} \sum_{k=0}^\infty \frac{\Gamma(\frac n2+k)\Gamma(m+n-1+k)}{\Gamma(k+1)\Gamma(k+n)} (1-t)^{k+n-1} \\
& \qquad \times\Big(\log(1-t)-\psi(k+1)-\psi(k+n)+\psi(\frac n2+k)+\psi(m+n+k-1)\Big).
\end{align*}
This can be rewritten as
\[ \begin{aligned}
S_m(t) &= \sum_{k=0}^{n-2}\frac{(m)_k(1-\frac n2)_k}{k!(2-n)_k}(1-t)^k \\
& \qquad +\frac{(-1)^n}{\Gamma(m)\Gamma(1-\frac n2)\Gamma(n-1)}
 \sum_{k=0}^\infty [\epsilon] \frac{\Gamma(\frac n2+k+\epsilon)\Gamma(m+n-1+k+\epsilon)}{\Gamma(k+1+\epsilon)\Gamma(k+n+\epsilon)}(1-t)^{k+n-1+\epsilon},
\end{aligned} \label{Erdelyiformula} \]
where $[\epsilon] f(\epsilon):=f'(0)$.

Replacing one of the $S_m$ in \eqref{YE} by \eqref{Erdelyiformula} and applying Lemma~\ref{Smlemma} term by term we obtain
\begin{align*}
I_m(s)&=A(s)+B(s),\\
\intertext{where}
A(s)&=\frac{\Gamma(n+s)(n-1)_m}{\Gamma(m+n+s)}\sum_{k=0}^{n-2}\frac{(m)_k (1-\frac n2)_k (s+1)_k (n+s)_k}{k!(2-n)_k (\frac n2+s+1)_k (m+n+s)_k}, \\
B(s)&=\frac{(-1)^n(n-1)_m\Gamma(\frac n2+s+1)}{\Gamma(m)\Gamma(1-\frac n2)\Gamma(n-1)\Gamma(s+1)} \\
& \qquad \times \sum_{k=0}^\infty [\epsilon] \frac{\Gamma(\frac n2+k+\epsilon) \Gamma(m+n-1+k+\epsilon)\Gamma(s+n+k+\epsilon)\Gamma(2n+s-1+k+\epsilon)}
  {\Gamma(k+1+\epsilon)\Gamma(k+n+\epsilon)\Gamma(\frac{3n}2+s+k+\epsilon)\Gamma(m+2n+s-1+k+\epsilon)}.
\end{align*}
It is easy to check that the $k$-th term of the series in $B(s)$ behaves like $k^{-2}$ for large~$k$,
hence the series converges absolutely for all values of $s$ and $m$
(as~long as we avoid the singularities of the $\Gamma$ functions, of course).

Using the fact that
$$ [\epsilon] \Gamma(a+\epsilon)=\Gamma(a)^2[\epsilon]\frac{1}{\Gamma(a-\epsilon)}, $$
we can write $A(s)$, $B(s)$ in the more illuminating form
\begin{align*}
A(s) &= \Gamma(n+s)(n-1)_m\Gamma(\frac n2+s+1)\sum_{k=0}^{n-2}\frac{(m)_k (1-\frac n2)_k (s+1)_k (n+s)_k}{k!(2-n)_k \Gamma(\frac n2+s+1+k)\Gamma(m+n+s+k)}, \\
B(s) &= \frac{(-1)^n(n-1)_m\Gamma(\frac n2+s+1)\Gamma(s+n)\Gamma(2n+s-1)^2 (s+1)_{n-1}}{\Gamma(m)\Gamma(1-\frac n2)\Gamma(n-1)} \\
& \qquad \times \sum_{k=0}^\infty [\epsilon] \frac{\Gamma(\frac n2+k+\epsilon)\Gamma(m+n-1+k+\epsilon)}
 {\Gamma(k+1+\epsilon)\Gamma(k+n+\epsilon)\Gamma(\frac{3n}2+s+k+\epsilon)\Gamma(m+2n+s-1+k+\epsilon)} \\
& \qquad \times \frac{(s+n)_k^2 (2n+s-1)_k^2}{\Gamma(s+n+k-\epsilon)\Gamma(2n+s+k-1-\epsilon)},
\end{align*}
which reveals that both $A$ and $B$ are meromorphic. In fact,
\begin{align*}
A(s) &= \Gamma(n+s)\Gamma(\frac n2+s+1) \times (\text{an entire function of $s$}) ,  \\
B(s) & =\Gamma(n+s)\Gamma(\frac n2+s+1)\Gamma(2n+s-1)^2 \times (\text{an entire function of $s$}) .
\end{align*}
Thus the function $I_m(s)$ can have a pole only at $s=-n-j$ or $s=-1-n/2-j$ or $s=1-2n-j$, $j\in\NN$.

Using the fact that
$$ \lim_{s\to -n}(n+s)\Gamma(n+s)= 1, $$
we can directly compute
$$ \lim_{s\to -n}(s+n) A(s) =(n-1)_m\sum_{k=0}^{n-2}\frac{(m)_k (1-\frac n2)_k (1-n)_k (0)_k}{k!(2-n)_k (1-\frac n2)_k \Gamma(m+k)} = \frac{(n-1)_m}{\Gamma(m)}.  $$
Since for $n$ even $B(s)$ vanishes identically (due~to the presence of $\Gamma(1-n/2)$ in the denominator) this proves the proposition in the case of $n>2$ even.

For $n>2$ odd, a~similar computation yields
\begin{align*}
\lim_{s\to -n}(s+n)B(s) & =\frac{-(n-1)_m\Gamma(1-\frac n2)\Gamma(n-1)^2 \Gamma(n)} {\Gamma(m)\Gamma(1-\frac n2)\Gamma(n-1)} \\
& \qquad \times \sum_{k=0}^\infty [\epsilon] \frac{\Gamma(\frac n2+k+\epsilon)\Gamma(m+n-1+k+\epsilon)}
 {\Gamma(k+1+\epsilon)\Gamma(k+n+\epsilon)\Gamma(\frac n2+k+\epsilon)\Gamma(m+n-1+k+\epsilon)} \\
& \qquad \times \frac{(0)_k^2 (n-1)_k^2}{\Gamma(k-\epsilon)\Gamma(n+k-1-\epsilon)}.
\end{align*}
The resulting series has only a single non-zero term for $k=0$. And even for this term we must spend the derivative with respect to $\epsilon $ on the factor
$$ \frac{1}{\Gamma(k-\epsilon)}, $$
otherwise we will get zero. Since
$$ \lim_{\epsilon\to 0} \frac{\Gamma'(-\epsilon)}{\Gamma(-\epsilon)^2}\to -1 ,  $$
we~thus obtain
$$ \lim_{s\to -n}(s+n)B(s) = \frac{(n-1)_m\Gamma(n-1)\Gamma(n)}{\Gamma(m)}
 \frac{\Gamma(\frac n2)\Gamma(m+n-1)}{\Gamma(1)\Gamma(n)\Gamma(\frac n2)\Gamma(m+n-1)}
 \frac1{\Gamma(n-1)} = \frac{(n-1)_m}{\Gamma(m)}.  $$
Combining both results we get
$$ \lim_{s\to -n}(n+s)I_m(s) = \lim_{s\to -n}(n+s) (A(s)+B(s)) = 2\frac{(n-1)_m}{\Gamma(m)}, $$
completing the proof for $n>2$ odd.  \end{proof}

It~was shown in Theorem~3.1 in~\cite{Ur} that
\[ I_m(s)\asymp(m+1)^{-s-1} \label{YF}  \]
for any fixed $s>-1$. We~will need a similar result for the somewhat more general integral
\[ I_{m,k}(s) := \frac{\Gamma(\frac n2+s+1)}{\Gamma(\frac n2)\Gamma(s+1)} \int_0^1
 [(2t\partial_t)^k (t^{m/2} S_m(t))]^2 \; t^{\frac n2-1} (1-t)^s \,dt ,  \label{YMK} \]
which reduces to \eqref{YE} for $k=0$.

\begin{proposition} \label{PZ}
For a fixed $s>-1$ and $k\in\{0,1,\dots,n-2\}$,
\[ I_{m,k}(s) \asymp (m+1)^{-s-1+2k} .  \label{YMI} \]
\end{proposition}

Actually, we~even have the following complete asymptotic expansion, of~which \eqref{YMI} is just the leading order term.

\begin{proposition} \label{prop2}
Assume that $s>-1$ and $k<n-1$. Then there exist constants $A_j$ depending only on $s$, $k$ and $n$ such that
$$ I_{m,k}(s) \approx \sum_{j=0}^\infty \frac{A_j}{m^{s+1-2k+j}} \qquad \text{as }m\to \infty,  $$
with $A_0>0$.
\end{proposition}

For the proof we are going to need several lemmas.
\begin{lemma} Let
$$ c_{j,k}(m) := (m)_j\frac{\Delta_-^j}{j!}(2y-m)^k|_{y=0},  $$
where $\Delta_-$ is the backward difference operator, i.e.
$$ \Delta_{-}f(y) = f(y)-f(y-1).  $$
Then
$$ I_{m,k}(s) = \frac{\Gamma(1+\frac n2+s)(n-1)_m^2}{\Gamma(\frac n2)\Gamma(s+1)(\frac n2)_m^2} \sum_{j,l=0}^k c_{j,k}(m) c_{l,k}(m) \tilde I_{j,l}, $$
where
$$ \tilde I_{j,l} := \int_0^1 t^{m+\frac n2-1}(1-t)^s \FF21{m+j,1-\frac n2}{m+\frac n2}t \FF21{m+l,1-\frac n2}{m+\frac n2}t \,dt.  $$
\end{lemma}

\begin{proof} The proof is based on the fact that for all~$x$:
\[ (2x)^k = \sum_{j=0}^k \frac{c_{j,k}(m)}{(m)_j} \Big(x+\frac m2\Big)_j ,  \label{PB30} \]
and the fact that
$$ \Big(t\partial_t+\frac m2\Big)_j t^{\frac m2} \FF21{m,1-\frac n2}{m+\frac n2}t
 = (m)_j t^{\frac m2} \FF21{m+j,1-\frac n2}{m+\frac n2}t. $$
\end{proof}

\begin{remark} \label{asyc}
Replacing $x$ by $mx+\frac m2$ in \eqref{PB30}, dividing by $m^k$ and letting $m\to \infty$ shows that
\[ c_{j,k}(m) \sim m^k \binom kj (-1)^{k-j} 2^j \qquad \text{as } m\to\infty. \label{PBc}  \]
With a little more effort, it~is possible to obtain in the same way the complete asymptotic expansion of $c_{j,k}(m)$
in decreasing powers of~$m$. \qed \end{remark}

\begin{lemma} Let $\sigma_1:=c+d-a-b$, $\sigma_2:=\gamma-\alpha-\beta$, $\sigma:=\sigma_1+\sigma_2$. Then
\begin{align} \label{Lemma3eq}
&\int_0^1 t^{c-1}(1-t)^{d-1} \FF21{a,b}ct \FF21{\alpha,\beta}\gamma t \,dt
 = \GG{c,\gamma,d,\sigma_1,\sigma}{\gamma-\beta,\sigma_1+a,\sigma_1+b,\sigma+\beta} \\
& \nonumber\qquad \times \sum_{k=0}^\infty \frac{(\beta)_k (\sigma_1)_k (\sigma_1+a-\alpha)_k (d)_k}{(\sigma+\beta)_k (\sigma_1+a)_k (\sigma_1+b)_k k!}
 \FF32{\beta+k,\sigma_1+k,c-a}{\sigma+\beta+k,\sigma_1+b+k}1,
\end{align}
whenever both sides makes sense.
\end{lemma}

\begin{remark}
A sufficient condition for the integral in \eqref{Lemma3eq} to converge is $c>0$, $\sigma_1>d>0$, $\sigma_2>0$.

A sufficient condition for the right-hand side to converge is $\gamma>\beta$, $\sigma_2+d>0$ (except for pathological values of parameters).

The condition $\sigma_2+d>0$ is needed for the value ${}_3 F_2(1)$ to exist, and $\gamma>\beta$ ensures that the whole series converges.
This can be seen employing the well known formula \cite[16.4.11]{NIST}
\[ \label{3F2trans} \FF32{a_1,a_2,a_3}{c_1,c_2}1 = \GG{c_2,c_1+c_2-a_1-a_2-a_3}{c_2+c_1-a_1-a_2,c_2-a_3} \FF32{a_3,c_1-a_1,c_1-a_2}{c_1,c_1+c_2-a_1-a_2}1. \]
Hence in our case
$$ \FF32{\beta+k,\sigma_1+k,c-a}{\sigma+\beta+k,\sigma_1+b+k}1 = \GG{\sigma_1+b+k,\sigma_2+d}{\sigma+b,d+k} \FF32{\sigma,\sigma_2+\beta,c-a}{\sigma+\beta+k,\sigma+b}1. $$
Therefore
$$ \FF32{\beta+k,\sigma_1+k,c-a}{\sigma+\beta+k,\sigma_1+b+k}1 \sim \frac{\Gamma(\sigma_2+d)}{\Gamma(\sigma+b)} k^{c-a} \qquad \text{as }k\to \infty, $$
by the known large parameter asymptotics of $\,_3 F_2$. See \cite[16.11.10]{NIST}.
Taking into account the behavior of all the Pochhammer symbols in the series,
we obtain that the $k$-th term behaves like $k^{\beta-\gamma-1}$.
Thus, indeed, $\gamma>\beta$ is sufficient for convergence.  \qed \end{remark}

\begin{proof} Denote the right hand side of \eqref{Lemma3eq} by~$I$.
First we represent the $\!_3 F_2$ function in $I$ by a double integral:
\begin{align*}
\FF32{\beta+k,\sigma_1+k,c-a}{\sigma+\beta+k,\sigma_1+b+k}1
&= \frac{(\sigma+\beta)_k (\sigma_1+b)_k}{(\beta)_k (\sigma_1)_k} \GG{\sigma+\beta,\sigma_1+b}{\beta,\sigma_1,\sigma,b} \\
& \qquad \times \int_0^1 \int_0^1 x^{\beta+k-1} (1-x)^{\sigma-1} y^{\sigma_1+k-1} (1-y)^{b-1} (1-xy)^{a-c} \,dx \,dy,
\end{align*}
and then swap the order of integration and summation so that $I$ becomes
\begin{align*}
I &= \GG{c,\gamma,d}{\gamma-\beta,\sigma_1+a,\beta,b} \\
& \qquad\times \int_0^1 \int_0^1 x^{\beta-1} (1-x)^{\sigma-1} y^{\sigma_1-1} (1-y)^{b-1} (1-xy)^{a-c}
 \FF21{\sigma_1+a-\alpha,d}{\sigma_1+a}{xy} \,dx\,dy.
\end{align*}
Next we represent the $\!_2 F_1$ function by a single integral:
$$ \FF21{\sigma_1+a-\alpha,d}{\sigma_1+a}{xy} = \GG{\sigma_1+a}{\alpha,\sigma_1+a-\alpha}
 \int_0^1 t^{\sigma_1+a-\alpha-1} (1-t)^{\alpha-1} (1-txy)^{-d} \,dt, $$
yielding
\begin{align*}
I &= \GG{c,\gamma,d}{\gamma-\beta,\beta,b,\alpha,\sigma_1+a-\alpha} \\
& \qquad \times \int_0^1 \int_0^1 \int_0^1 x^{\beta-1} (1-x)^{\sigma-1} y^{\sigma_1-1} (1-y)^{b-1}
 t^{\sigma_1+a-\alpha-1} (1-t)^{\alpha-1} (1-xy)^{a-c}(1-txy)^{-d} \,dt\,dx\,dy.
\end{align*}
Integrating with respect to the $y$ variable produces
\begin{align*}
I &= \GG{c,\gamma,d,\sigma_1}{\gamma-\beta,\beta,\sigma_1+b,\alpha,\sigma_1+a-\alpha} \\
& \qquad \times \int_0^1 \int_0^1 x^{\beta-1} (1-x)^{\sigma-1} t^{\sigma_1+a-\alpha-1} (1-t)^{\alpha-1}
 \FE1{\sigma_1;c-a,d}{\sigma_1+b}{x,tx} \,dt\,dx,
\end{align*}
where $F_1$ is the first Appell hypergeometric function.
The~well known transformation rule \cite[16.16.1]{NIST}
$$ \FE1{a;b_1,b_2}{b_1+b_2}{x,y} =(1-x)^{-a} \FF21{a,b_2}{b_1+b_2}{\frac{x-y}{x-1}} $$
implies
$$ \FE1{\sigma_1;c-a,d}{\sigma_1+b}{x,tx} = (1-x)^{-\sigma_1} \FF21{\sigma_1,d}{\sigma_1+b}{\frac x{x-1}(1-t)}. $$
Substituting this into the last formula for $I$ we get
\begin{align*}
I &= \GG{c,\gamma,d,\sigma_1}{\gamma-\beta,\beta,\sigma_1+b,\alpha,\sigma_1+a-\alpha} \\
& \qquad \times \int_0^1 \int_0^1 x^{\beta-1} (1-x)^{\sigma_2-1} t^{\sigma_1+a-\alpha-1} (1-t)^{\alpha-1}
 \FF21{\sigma_1,d}{\sigma_1+b}{\frac x{x-1}(1-t)} \,dt\,dx.
\end{align*}
This can be integrated with respect to~$t$:
\begin{align*}
I &= \GG{c,\gamma,d,\sigma_1}{\gamma-\beta,\beta,\sigma_1+a,\sigma_1+b} \int_0^1 x^{\beta-1} (1-x)^{\sigma_2-1}
 \FF32{\sigma_1,d,\alpha}{\sigma_1+b,\sigma_1+a}{\frac x{x-1}} \,dx.
\end{align*}
Now we employ Lemma 2.10 in \cite{Ur} which asserts that
\begin{align*}
\int_0^1 t^{c-1}(1-t)^{d-1} \FF21{a,b}ct (1-tx)^{-\alpha} \,dt
&= \GG{c,d,\sigma_1}{\sigma_1+a,\sigma_1+b} (1-x)^{-\alpha} \\
& \qquad \times  \FF32{\sigma_1,d,\alpha}{\sigma_1+b,\sigma_1+a}{\frac x{x-1}}.
\end{align*}
Using this we get
\begin{align*}
I &= \GG{\gamma}{\gamma-\beta,\beta} \int_0^1 \int_0^1 x^{\beta-1} (1-x)^{\gamma-\beta-1}
 t^{c-1} (1-t)^{d-1} \FF21{a,b}{c}t (1-tx)^{-\alpha} \,dt\,dx.
\end{align*}
A~final integration with respect to $x$ gives us what we want:
\begin{align*}
I &= \int_0^1 t^{c-1} (1-t)^{d-1} \FF21{a,b}ct \FF21{\alpha,\beta}\gamma t \,dt.
\end{align*}

This proves \eqref{Lemma3eq} for values of $a$, $b$, $c$, $d$, $\alpha$, $\beta$, $\gamma$
for which all the operations above make sense and all integrals and series converge uniformly.
But since both sides of \eqref{Lemma3eq} are meromorphic functions (of~all parameters) in their
respective domains, we~can extend the validity of this argument by analytic continuation.
\end{proof}

\begin{corollary} \label{cor1}
For $n+s>l$ and $m+n>1$,
\begin{align} \label{Cor1eq}
\frac{(n-1)_m^2}{(\frac n2)_m^2} \tilde I_{j,l}
&=\GG{m+n-1}{m+n+s} \GG{\frac n2,\frac n2,s+1,n+s-j,2n+s-1-j-l}{n-1,n-1,\frac n2+s+1-j,\frac{3n}{2}+s-j-l} \\
& \nonumber \qquad \times \sum_{k=0}^\infty \frac{(1-\frac n2)_k (n+s-j)_k (n+s-l)_k (s+1)_k}{(\frac{3n}2+s-j-l)_k (m+n+s)_k (\frac n2+s+1-j)_k k!} \\
& \hskip6em \nonumber \times \FF32{1-\frac n2+k,n+s-j+k,\frac n2-j}{\frac{3n}2+s-j-l+k,\frac n2+s+1-j+k}1.
\end{align}
\end{corollary}

\begin{proof} Apply the last lemma to the integral defining $\tilde I_{j,l}$. \end{proof}

\begin{corollary} For $n+s>l$, $n+s>j$ and $s>-1$, as $m\to \infty$:
\begin{align*}
\frac{(n-1)_m^2}{(\frac n2)_m^2} \tilde I_{j,l}
&\sim m^{-s-1} \GG{\frac n2,\frac n2,s+1,n+s-j,n+s-l}{n-1,n-1,\frac n2+s+1-j,\frac n2+s+1-l} \\
& \qquad \times \FF32{1-\frac n2,s+1,1-\frac n2}{\frac n2+s+1-j,\frac n2+s+1-l}1.
\end{align*}
\end{corollary}

\begin{proof}
Apply the formula \eqref{3F2trans} to the first term in~\eqref{Cor1eq}.
\end{proof}

\begin{corollary} \label{cor3}
For $s>-1$ and $k<n-1$, as $m\to\infty$:
\begin{align} \label{leadingorder}
I_{m,k}(s) &\sim m^{2k-s-1} \GG{\frac n2,1+\frac n2+s,n+s-k,n+s-k}{n-1,n-1,\frac n2+1+s-k,\frac n2+1+s-k}  \\
& \nonumber \qquad \times \sum_{j=0}^\infty \frac{(1-\frac n2)_j^2 (s+1)_j}{(1+\frac n2+s-k)_j^2 j!} \FF21{-k,1-\frac n2+j}{1+\frac n2+s-k+j}{-1} ^2.
\end{align}
\end{corollary}

\begin{proof}
Combining the previous corollary, the definition of $I_{m,k}(s)$ and the asymptotics \eqref{PBc}
of the coefficients $c_{j,k}(m)$, $c_{l,k}(m)$ we get as $m\to \infty$:
\begin{align*}
I_{m,k}(s) &\sim m^{2k-s-1} \GG{\frac n2,n+s,n+s}{n-1,n-1,\frac n2+1+s} \\
& \qquad \times \sum_{j,l=0}^k \frac{(-k)_j (-k)_l (-\frac n2-s)_j (-\frac n2-s)_l}{j!k! (1-n-s)_j (1-n-s)_l}
 2^{j+l} \FF32{1-\frac n2,s+1,1-\frac n2}{\frac n2+s+1-j,\frac n2+s+1-l}1.
\end{align*}
Note that the series for $\!_3 F_2(1)$ converges for $s+2n-2k-1>0$ which is satisfied due to our assumptions $s>-1$, $k<n-1$.

Expanding the $\!_3 F_2$ into series and using the fact that
$$ \frac{(-\frac n2-s)_j}{(1+\frac n2+s-j)_r} = \frac{(-\frac n2-s-r)_j}{(1+\frac n2+s)_r} $$
for all~$r$, we~obtain
\begin{align}  \label{pom1}
I_{m,k}(s) &\sim m^{2k-s-1} \GG{\frac n2,n+s,n+s}{n-1,n-1,\frac n2+1+s} \\
& \nonumber \qquad \times \sum_{r=0}^\infty \frac{(1-\frac n2)_r^2 (s+1)_r}{(1+\frac n2+s)_r^2 r!} \FF21{-k,-\frac n2-s-r}{1-n-s}2 ^2.
\end{align}
We~now use the transform \cite[15.8.6]{NIST} to~get
\begin{align*}
\FF21{-k,-\frac n2-s-r}{1-n-s}2 &= (-2)^k \GG{n+s-k,1+\frac n2+s}{n+s,1+\frac n2+s-k}  \\
& \qquad \times \frac{(1+\frac n2+s)_r}{(1+\frac n2+s-k)_r} \FF21{-k,n+s-k}{1+\frac n2+s-k+r}{\frac12},
\end{align*}
and then use the Pfaff transform \cite[\S2.1~(22)]{BE} in the form
$$ \FF21{-k,n+s-k}{1+\frac n2+s-k+r}{\frac12} = 2^{-k} \FF21{-k,1-\frac n2+r}{1+\frac n2+s-k+r}{-1}.  $$
Inserting this into \eqref{pom1} we get the right-hand side of \eqref{leadingorder}.
\end{proof}

\begin{remark}
It must be pointed out that the series in (\ref{leadingorder}) converges for $2n+s>2k+1$, but it is guaranteed
to be positive only for $s>-1$. For $s\leq-1$ we cannot (in~general) be sure that the right hand side of
\eqref{leadingorder} is nonzero, which is essential for the asymptotic equality to hold.

Here are some more details. Denote the right hand side of (\ref{leadingorder}) by $L_k(s)$,
so that for $s>-1$ we have by Corollary~\ref{cor3}
$$ I_{m,k}\sim L_k(s) \qquad\text{as }m\to \infty. $$
The function $L_k(s)$ can also be expressed as follows:
$$ L_k(s) = \frac{\Gamma(\frac n2) \Gamma(1+\frac n2+s)}{m^{s+1-2k}} \frac{\Gamma(n+s-k)^2}{\Gamma(n-1)^2}
 \sum_{j=0}^\infty \frac{(1-\frac n2)_j^2 (s+1)_j}{j!} \bFF21{-k,1-\frac n2+j}{1+\frac n2+s-k+j}{-1} ^2, $$
where
$$ \bFF21{a,b}cx := \frac1{\Gamma(c)} \FF21{a,b}cx = \sum_{k=0}^\infty \frac{(a)_k (b)_k}{\Gamma(c+k)k!}x^k $$
is the regularized hypergeometric function, which is an entire function in all three parameters~$a,b,c$.
% Since, in our case, the series for $\!_2{\bf F}_1$ terminates (due~to the presence of upper parameter $-k$), it~is clear that
% $$ \bFF21{-k,1-\frac n2+j}{1+\frac n2+s-k+j}{-1}  $$
% is an entire function of $s$.
Also, as discussed before, the series converges for $2n+s>2k+1$.
Therefore
$$ L_k(s) = m^{2k-s-1} \Gamma(1+\frac n2+s)\Gamma(n+s-k)^2\times(\text{a holomorphic function in $\Re s>2k+1-2n$}). $$
The function $L_k(s)$ is therefore holomorphic whenever $\Re s>2k+1-2n$ and $s$ avoids the singularities
of $\Gamma(1+\frac n2+s)\Gamma(n+s-k)$. We have that $L_k(s)>0$ for $s>-1$ but, indeed, for other values of $s$
the function $L_k(s)$ need not be positive in general. In~fact, it is easy to see that
$$ L_0(-2)=0, $$
and hence for $s=-2$ it is not true that $I_{m,0}\sim L_0(-2)$ as $m\to \infty$. \qed \end{remark}

\begin{remark}
For $k=0$ we get the simpler formula
$$ L_0(s) =  \frac{\Gamma(\frac n2)\Gamma(1+\frac n2+s)}{m^{s+1}} \frac{\Gamma(n+s)^2}{\Gamma(n-1)^2}
 \bFF32{1-\frac n2,1-\frac n2,s+1}{1+\frac n2+s,1+\frac n2+s}1, $$
valid for $\Re s>1-2n$. This agrees with the leading factor of $I_m$ computed in \cite{Ur} for~$s>-1$.
\qed \end{remark}

\begin{proof}[Proof of Proposition~\ref{prop2}]
The proof of Proposition~\ref{prop2} now follows upon combining Corollary~\ref{cor1} with
the asymptotic expansions of $c_{j,k}(m)$, $c_{l,k}(m)$ (which are left to the reader, cf.~Remark~\ref{asyc}).
Corollary~\ref{cor3} shows that, indeed, $A_0>0$.
\end{proof}

\section{The \Hh Dirichlet space}  \label{sec4}
As~in~\cite{EY2}, introduce the notation
\[ I_m^\circ := \lim_{s\to-n} (s+n) I_m(s) = \begin{cases} {(n-1)_m}/{\Gamma(m)} \qquad&n\text{ even}, \\
 {2(n-1)_m}/{\Gamma(m)} &n\text{ odd}, \end{cases}   \label{YG}  \]
and consider the space
$$ \HH_\circ := \{f=\sum_m f_m: \; f_m\in\bhm\forall m, \; \|f\|_\circ^2:=\sum_m I_m^\circ \|f_m\|^2_\pbn<+\infty\}. $$
The~quantity $\|f\|_\circ$ vanishes on constants, is a norm on $\HH_{\circ,0}:=\{f\in\HH_\circ: f(0)=0\}$,
and a seminorm on~$\HH_\circ$; in~other words, $\|f\|_\circ+|f(0)|$ is a norm on~$\HH_\circ$.

Let $X_{jk}$, $j,k=1,\dots,n$, $j\neq k$, denote the tangential vector fields
$$ X_{jk} := x_j\partial_k - x_k\partial_j   $$
on~$\RR^n$, and denote by $\XX_j$, $j=1,\dots,n(n-1)$, the collection of all these operators
(in~some fixed order), i.e.
$$ \{X_{jk}: \; j,k=1,\dots,n,\;j\neq k\} = \{\XX_j: \; j=1,\dots,n(n-1) \}. $$
By~a~routine computation, one~checks that
$$ \sum_{j,k=1}^n X_{jk}^2 = 2\dsph,  $$
where $\dsph$ is the spherical Laplacian on~$\RR^n$: for $x=r\zeta$ with $r>0$ and $\zeta\in\pbn$,
the~ordinary Euclidean Laplacian is expressed~as
$$ \Delta = \frac{\partial^2}{\partial r^2} + \frac{n-1}r \frac\partial{\partial r} + \frac1{r^2}\dsph.  $$
The~operator $\dsph$ commutes with the action of the orthogonal group $O(n)$ of~$\RR^n$,
hence it is automatically diagonalized by the Peter-Weyl decomposition~\eqref{WA}:
a~simple  computation reveals that
\[ \dsph|\bhm = -m(m+n-2)I|\bhm  \label{WK}  \]
where $I$ stands for the identity operator.

In~addition to the tangential operators $X_{jk}$, we~denote by
$$ \cR := \sum_j x_j \partial_j  $$
the radial derivative operator on~$\RR^n$; note that in the polar coordinates $x=r\zeta$, $\cR=r\partial_r$.
Finally, we denote by $\YY_j$, $0\le j\le n(n-1)$, the collection of operators $\YY_j:=\XX_j$ for $j\ge1$ and $\YY_0:=\cR$:
$$ \{\cR\} \cup \{X_{jk}: \; j,k=1,\dots,n,\;j\neq k\} = \{\YY_j: \; j=0,\dots,n(n-1) \}, \qquad \YY_0=\cR. $$

\begin{theorem} If~$f=\sum_m f_m$, $f_m\in\bhm$, is \Hh on $\bn$, $n>2$, then the following assertions are equivalent:
\begin{itemize}
\item[(a)] $f\in\HH_\circ$;
\item[(b)] $\sum_m m^{n-1} \|f_m\|^2_\pbn <+\infty$;
\item[(c)] for some (equivalently, any) real $p>\frac{n-1}2$, $(-\dsph)^{p/2}f\in L^2(\bn,d\rho_{2p-n})$, i.e.
\[ \|(-\dsph)^{p/2}f\|^2_{2p-n} < +\infty ;  \label{YHa} \]
\item[(d)] for some (equivalently, any) integer $p>\frac{n-1}2$,
\[ \sum_{j_1,\dots,j_p=1}^{n(n-1)} \|\XX_{j_1}\dots\XX_{j_p}f\|^2_{2p-n} < +\infty .  \label{VM}  \]
\end{itemize}
If $n>3$, then all the above are additionally equivalent~to
\begin{itemize}
\item[(e)] for some (equivalently, any) integer $p$ satisfying $n-2\ge p>\frac{n-1}2$,
\[ \sum_{q=0}^p \sum_{j_1,\dots,j_q=0}^{n(n-1)} \|\YY_{j_1}\dots\YY_{j_q}f\|^2_{2p-n} < +\infty ;  \label{VP}  \]
\item[(f)] for some (equivalently, any) integer $p$ satisfying $n-2\ge p>\frac{n-1}2$,
\[ \sum_{|\alpha|\le p} \|\partial^\alpha f\|^2_{2p-n} < +\infty. \label{VO} \]
\end{itemize}
If $n$ is odd, then {\rm(a)--(d)} are additionally equivalent~to
\begin{itemize}
\item[(g)] for $p=\frac{n-1}2$,
\[ \sup_{0<r<1} \sum_{j_1,\dots,j_p=1}^{n(n-1)} \|\XX_{j_1}\dots\XX_{j_p}f(r\cdot)\|^2_\pbn < +\infty .  \label{VN} \]
\end{itemize}
Furthermore, the~square roots of the quantities in {\rm(b), (c), (d)} and {\rm(g)} are equivalent to~$\|f\|_\circ$,
and of those in {\rm(e), (f)} are equivalent to~$\|f\|_\circ+|f(0)|$.
\end{theorem}

\begin{proof} (a)$\iff$(b) Clearly from~\eqref{YG}
\[ I^\circ_m \asymp m^{n-1},  \label{YH}  \]
and the claim is thus immediate from the definition of~$\|\cdot\|_\circ$.

\medskip

(b)$\iff$(d) Since the adjoint of $X_{jk}$ in $L^2(\pbn,d\sigma)$ is just~$-X_{jk}$,
we~have for any $g\in L^2(\pbn,d\sigma)$
$$ \sum_{j=1}^{n(n-1)} \|\XX_j g\|^2_\pbn = -\sum_{j,k=1}^n \spr{X_{jk}^2g,g}_\pbn = -2\spr{\dsph g,g}_\pbn, $$
so~for $g=\sum_m g_m$, $g_m\in\chm$, as~in~\eqref{WA},
\[ \sum_{j=1}^{n(n-1)} \|\XX_j g\|^2_\pbn = \sum_m 2m(m+n-2) \|g_m\|^2_\pbn ,  \label{WL} \]
by~\eqref{WK}. Iterating this procedure, we~get
$$ \sum_{j_1,\dots,j_p=1}^{n(n-1)} \|\XX_{j_1}\dots\XX_{j_p}g\|^2_\pbn = \sum_m [2m(m+n-2)]^p \|g_m\|^2_\pbn . $$
Applying this now to $g(\zeta)=f(r\zeta)$ where $f$ is \Hh on~$\bn$, we~obtain by~\eqref{YC}
\[ \sum_{j_1,\dots,j_p=1}^{n(n-1)} \|\XX_{j_1}\dots\XX_{j_p}f(r\cdot)\|^2_\pbn = \sum_m [2m(m+n-2)]^p r^{2m} S_m(r^2)^2 \|f_m\|^2_\pbn , \label{YI} \]
and, for any $s>-1$,
\begin{align*}
& \sum_{j_1,\dots,j_p=1}^{n(n-1)} \|\XX_{j_1}\dots\XX_{j_p}f\|^2_s \\
& \hskip4em = \frac{\Gamma(\frac n2+s+1)}{\pi^{n/2}\Gamma(s+1)} \int_0^1 \frac{2\pi^{n/2}}{\Gamma(\frac n2)}
 \sum_{j_1,\dots,j_p=1}^{n(n-1)} \|\XX_{j_1}\dots\XX_{j_p}f(r\cdot)\|^2_\pbn \,(1-r^2)^s r^{n-1}\,dr \\
& \hskip4em = \frac{\Gamma(\frac n2+s+1)}{\Gamma(\frac n2)\Gamma(s+1)}
 \int_0^1 \sum_m [2m(m+n-2)]^p \|f_m\|^2_\pbn t^{m+\frac n2-1}S_m(t)^2(1-t)^s \,dt \\
& \hskip4em = \sum_m [2m(m+n-2)]^p I_m(s) \|f_m\|^2_\pbn \qquad\text{by \eqref{YE}}.
\end{align*}
Recalling~\eqref{YF}, we~see that if $s=2p-n$, then for all $m\ge1$
\[ [2m(m+n-2)]^p I_m(s) \asymp m^{2p-s-1} = m^{n-1} , \label{YK} \]
while for $m=0$ both sides vanish. This proves the claim.

\medskip

(b)$\iff$(c) First of all, from the expression of $d\rho_s$ in polar coordinates
$$ d\rho_s(x) = \frac{\Gamma(\frac n2+s+1)}{\Gamma(\frac n2)\Gamma(s+1)} r^{n-1}(1-r^2)^s
 \frac{2\pi^{n/2}}{\Gamma(\frac n2)} \, d\sigma(\zeta) \,dr, \qquad x=r\zeta, \;r>0, \zeta\in\pbn,  $$
we~get the expression as Hilbert space tensor product
$$ L^2(\bn,d\rho_s) = L^2((0,1),\tfrac{\Gamma(\frac n2+s+1)}{\Gamma(\frac n2)\Gamma(s+1)} r^{n-1}(1-r^2)^s\,dr)
 \otimes L^2(\pbn, \tfrac{2\pi^{n/2}}{\Gamma(\frac n2)} \, d\sigma(\zeta)) .   $$
By~\eqref{WA} and~\eqref{WK}, it~follows that $\dsph$ then acts~as
$$ I \otimes \bigoplus_m [-m(m+n-2)]I|\chm ,  $$
and, hence, $-\dsph$ is indeed a positive selfadjoint operator on each $L^2(\bn,d\rho_s)$, $s>-1$,
and the power $(-\dsph)^q$ makes sense for any real $q\ge0$ by the spectral theorem:
$$ (-\dsph)^q \sum_m F_m(r) f_m(\zeta) = \sum_m [m(m+n-2)]^q F_m(r) f_m(\zeta)  $$
for any $f_m\in\chm$ and $F_m\in L^2((0,1),r^{n-1}(1-r^2)^s\,dr)$.

Using again \eqref{YC}, the~claim thus follows by the same calculation as from~\eqref{YI} above.

\medskip

(b)$\iff$(g) Observe from \eqref{YI}, combined with the fact that $r^m S_m(r^2)\nearrow1$ as $r\nearrow1$,
that the supremum on the left-hand side of \eqref{VN} equals
$$ \sum_m [2m(m+n-2)]^p \|f_m\|^2_\pbn $$
by the Lebesgue Monotone Convergence Theorem. Since now $2p=n-1$, the~claim thus follows in the same way
as in~\eqref{YK}.

\medskip

(b)$\iff$(e) Observe first of all that the additional hypothesis $n>3$ ensures that
there exists at least one integer $p$ in the interval $n-2\ge p>\frac{n-1}2$.
Let
$$ f(r\zeta) = \sum_m r^m S_m(r^2) f_m(\zeta), \qquad 0\le r<1, \zeta\in\pbn,   $$
be the Peter-Weyl decomposition of our \Hh function~$f$. Then
$$ \XX_{j_1}\dots\XX_{j_l} \cR^k f (r\zeta) = \sum_m \cR^k[r^m S_m(r^2)] \XX_{j_1}\dots\XX_{j_l}f_m(\zeta) $$
(which also remains in force for any permutation of the order of the $\XX_j$ and $\cR$ on the left-hand side).
Hence, similarly as for~\eqref{YI},
$$ \sum_{j_1,\dots,j_q=1}^{n(n-1)} \|\XX_{j_1}\dots\XX_{j_q}\cR^k f(r\cdot)\|^2_\pbn
 = \sum_m [2m(m+n-2)]^q (\cR^k[r^m S_m(r^2)])^2 \|f_m\|^2_\pbn , $$
% By~the Leibnitz rule,
% \begin{align*}
% \cR^k[r^m S_m(r^2)] &= \sum_{j=0}^k \binom kj \cR^{k-j}(r^m) \cR^j(S_m(r^2)) \\
% &= \sum_{j=0}^k \binom kj m^{k-j}r^m \cR^j(S_m(r^2)) ,
% \end{align*}
% and similarly
% $$ \cR^j (S_m(r^2) = \sum_{i=1}^j P_{ij}(r) S_m^{(i)}(r^2) ,  $$
% where $P_{ij}$ are some polynomials (with coefficients depending only on~$i,j$),
% while, by \cite[\S2.1, (7) and (23)]{BE},
% \begin{align*}
% S_m^{(i)}(t) &= \frac{(n-1)_m}{(\frac n2)_m} \frac{(m)_i(1-\frac n2)_i}{(m+\frac n2)_i} \FF21{m+i,1-\frac n2+i}{m+\frac n2+i}t \\
% &= \frac{(n-1)_m}{(\frac n2)_m} \frac{(m)_i(1-\frac n2)_i}{(m+\frac n2)_i} (1-t)^{n-1-i} \FF21{\frac n2,m+n-1}{m+\frac n2+i}t .
% \end{align*}
% for $i\le n-2$,
and, for any $s>-1$,
\begin{align*}
& \sum_{j_1,\dots,j_q=1}^{n(n-1)} \|\XX_{j_1}\dots\XX_{j_q}\cR^k f\|^2_s \\
& \hskip4em = \frac{2\Gamma(\frac n2+s+1)}{\Gamma(\frac n2)\Gamma(s+1)}
 \sum_m [2m(m+n-2)]^q \|f_m\|^2_\pbn \int_0^1 (\cR^k[r^m S_m(r^2)])^2 r^{n-1} (1-r^2)^s \,dr \\
& \hskip4em = \frac{\Gamma(\frac n2+s+1)}{\Gamma(\frac n2)\Gamma(s+1)}
 \sum_m [2m(m+n-2)]^q \|f_m\|^2_\pbn \int_0^1 ((2t\partial_t)^k[t^{m/2} S_m(t)])^2 t^{\frac n2-1} (1-t)^s \,dt \\
& \hskip4em = \sum_m [2m(m+n-2)]^q I_{m,k}(s) \|f_m\|^2_\pbn
\end{align*}
by~\eqref{YMK}. Using~\eqref{YMI}, we~again see that if $q+k\le p<n-1$ and $s=2p-n$, then for all $m\ge1$
\[ [2m(m+n-2)]^q I_{m,k}(s) \asymp m^{2q-s-1+2k} \le m^{2p-s-1} = m^{n-1} \asymp I^\circ_m; \label{YMS} \]
% while for $m=0$ both sides vanish;
and if $q+k=p<n-1$ and $s=2p-n$, then for all $m\ge1$ even
\[ [2m(m+n-2)]^q I_{m,k}(s) \asymp m^{2q-s-1+2k} = m^{2p-s-1} = m^{n-1} \asymp I^\circ_m. \label{YMT} \]
% while for $m=0$ both sides vanish.
For $m=0$, \eqref{YMT}~still holds, and so does \eqref{YMS} if $0<q+k$;
for $m=q=k=0$, the~right-hand side of \eqref{YMS} becomes 1 while the left-hand side is zero.
Consequently, for $q<p<n-1$,
$$ \sum_{j_1,\dots,j_q=0}^{n(n-1)} \|\YY_{j_1}\dots\YY_{j_q}f\|^2_{2p-n} \lesssim \|f\|_\circ^2 + \delta_{q0}|f(0)|^2,  $$
while for $q=p<n-1$ even
$$ \sum_{j_1,\dots,j_q=0}^{n(n-1)} \|\YY_{j_1}\dots\YY_{j_q}f\|^2_{2p-n} \asymp \|f\|_\circ^2.  $$
The claim follows.

\medskip

(f)$\implies$(e) By~the Leibnitz rule, $\XX_{j_1}\dots\XX_{j_p}f(x)=\sum_{|\alpha|\le p} P_\alpha(x) \partial^\alpha f(x)$,
with some coefficient functions $P_\alpha$ that are bounded on $\bn$ (in~fact --- they are polynomials).
The~claim is thus immediate from the triangle inequality.

\medskip

(e)$\implies$(f) Observe that at any $x\in\bn$, $x\neq0$, the tangential vector fields $\XX_j$ span
(very redundantly) the entire tangent space to the sphere~$|x|\pbn$; thus together with~$\cR$,
they span the whole tangent space at~$x$. It~follows that, for any $q\in\NN$, the~derivatives
$\partial^\alpha f$, $|\alpha|=q$, can be expressed as linear combinations of the derivatives
$\YY_{j_1}\dots\YY_{j_q}f$, $1\le j_1,\dots,j_q\le n(n-1)$; furthermore, the~coefficients of
these linear combinations can be chosen to be bounded away from the origin $x=0$ (where all the
$\XX_j$ as well as $\cR$ vanish). In~other words, if~we momentarily denote by $\chi$ the characteristic
function of the annular region $\frac14<|x|<1$, then for any $s>-1$ and~$q\in\NN$,
$$ \sum_{|\alpha|=q} \|\chi\partial^\alpha f\|_s^2 \lesssim \sum_{j_1,\dots,j_q=1}^{n(n-1)} \|\YY_{j_1}\dots\YY_{j_q}f\|^2_s , $$
and, hence, for any $p\in\NN$,
\[ \sum_{q=0}^p \sum_{|\alpha|=q} \|\chi\partial^\alpha f\|_s^2 \lesssim
 \sum_{q=0}^p \sum_{j_1,\dots,j_q=1}^{n(n-1)} \|\YY_{j_1}\dots\YY_{j_q}f\|^2_s . \label{YNA} \]
To~treat the remaining region $|x|<1/4$, we~use a ``subharmonicity'' argument. Recall that \Hh
functions possess the mean-value property~\cite[Corollary~4.1.3]{St}
$$ f(a) = \int_\pbn f(\phi_a(r\zeta)) \,d\sigma(\zeta)   $$
for any $r\in(0,1)$, $a\in\bn$ and $f$ \Hh on~$\bn$. Integrating over $0<r<\frac14$ yields
$$ f(a) = c_n \int_{|x|<1/4} f(\phi_a(x)) \,dx ,  $$
where $c_n:=1/\int_{|x|<1/4}\,dx$. Changing the variable $x$ to~$\phi_a(x)$ gives
$$ f(a) = \int_{|\phi_a(x)|<1/4} f(x) F(a,x) \, d\rho_s(x),  $$
with
$$ F(a,x):=c_n \frac{\pi^{n/2}\Gamma(s+1)}{\Gamma(\frac n2+s+1)} \operatorname{Jac}_{\phi_a}(x) (1-|x|^2)^{-s} $$
smooth on~$\bn\times\bn$. Hence for any multiindex~$\alpha$,
$$ \partial^\alpha f(a) = \int_{|\phi_a(x)|<1/4} f(x) \;\partial^\alpha_a F(a,x) \, d\rho_s(x).  $$
If~$|a|<\frac14$, one easily checks from \eqref{YNC} below that $|\phi_a(x)|<\frac14$ implies
$$ 1-|x|^2 = 1-|\phi_a(\phi_a(x))|^2 = \frac{(1-|a|^2)(1-|\phi_a(x)|^2)}{1-2\spr{a,\phi_a(x)}+|a|^2|\phi_a(x)|^2}
 \ge \frac{(1-4^{-2})^2}{(1+4^{-2})^2} = \frac{15^2}{17^2} ,  $$
or $|x|<8/17$. Since $\partial^\alpha_a F(a,x)$~is, thanks to the smoothness of~$F$ on~$\bn\times\bn$,
bounded on $|a|<\frac14$ and $|x|<\frac8{17}$, we~thus get
$$ |\partial^\alpha f(a)| \le C_\alpha \int_{|\phi_a(x)|<1/4} |f(x)| \, d\rho_s(x) \le C_\alpha \|f\|_s $$
with some finite $C_\alpha$ independent of $|a|<1/4$ and~$f$. Hence
$$ \|(1-\chi)\partial^\alpha f\|^2_s \le C_\alpha^2 \|1-\chi\|^2_s \|f\|^2_s,  $$
and, consequently,
$$ \sum_{q=0}^p \sum_{|\alpha|=q} \|(1-\chi)\partial^\alpha f\|_s^2 \lesssim \|f\|^2_s .  $$
Combining this with \eqref{YNA} and setting $s=2p-n$, the claim follows.

\medskip

This completes the proof of the theorem.  \end{proof}

Part (b) of the last theorem can be reformulated as follows. Consider the weak-maximal operator $X$ acting
from $L^2(\pbn)$ into the Cartesian product of $[n(n-1)]^p$ copies of $L^2(\pbn)$ by
$$ g\longmapsto \{\XX_{j_1}\dots\XX_{j_p}g\}_{j_1,\dots,j_p=1}^{n(n-1)} ; $$
that~is, the~domain of $X$ consists of all $g\in L^2(\pbn)$ for which all the $\XX_{j_1}\dots\XX_{j_p}g$
exist in the sense of distributions and belong to $L^2(\pbn)$.
(In~other words, $X=Y^*$ where $Y$ is the restriction of the formal adjoint $X^\dagger$
of $X$ to $\bigoplus^{[n(n-1)]^p} C^\infty(\pbn)$.)
Then~$f\in\HH$ belongs to $\HH_\circ$ if and only if $f(r\cdot)\to f^*$ as $r\nearrow1$ in $L^2(\pbn)$
for some ``boundary value'' $f^*$ of~$f$, and $f^*\in\operatorname{Dom}(X)$.
Furthermore, $\|Xf^*\|$ is a seminorm equivalent to~$\|f\|_\circ$.
(The~reader is referred to Grellier and Jaming \cite[Theorem~A]{GJ} for much more detailed discussion
of the matters above and boundary values of~\Hh functions in general.)

We~use this reformulation in the proof of the next proposition for $n=3$,
writing for ease of notation just $f$ instead of~$f^*$.

\section{Moebius invariance} \label{sec5}
\begin{proposition} The space $\HH_\circ$ is Moebius invariant:
$f\in\HH_\circ$ implies $f\circ U\in\HH_\circ$, $f\circ\phi_a\in\HH_\circ$ for any $U\in\On$ and $a\in\bn$.
Also, the composition operators $f\mapsto f\circ U$, $f\mapsto f\circ\phi_a$
are continuous on~$\HH_\circ$.
\end{proposition}

\begin{proof} For $\On$-invariance, both assertions are immediate from \eqref{WA}
and the definition of inner product in~$\HH_\circ$ --- in~fact, the composition operator
$f\mapsto f\circ U$, $U\in\On$, is~unitary.

So~consider the composition with~$\phi_a$. Assume first that $n>3$.
By~part (f) of the last theorem, it~is then enough to show that
\[ \sum_{|\alpha|\le p} \|\partial^\alpha (f\circ\phi_a)\|^2_{2p-n}
 \lesssim \sum_{|\beta|\le p} \|\partial^\beta f\|^2_{2p-n}  \label{YNB}  \]
for some integer $p$ satisfying $n-2\ge p>\frac{n-1}2$ (thanks to the assumption that $n>3$,
such an integer exists). However, by~the chain rule, we~have
$$ \partial^\alpha(f\circ\phi_a)(x) = \sum_{|\beta|\le|\alpha|} c_{\alpha\beta}(x) (\partial^\beta f)(\phi_a(x)),  $$
where the coefficient functions $c_{\alpha\beta}$ are polynomials in the derivatives of~$\phi_a$;
in~particular, they are bounded on~$\bn$. Furthermore, the Jacobian $\operatorname{Jac}_{\phi_a}(x)$
is likewise bounded on~$\bn$ for any fixed~$a\in\bn$. Consequently, for any $s>-1$,
$$ \|\partial^\alpha(f\circ\phi_a)\|^2_s \lesssim \sum_{|\beta|\le|\alpha|} \|\partial^\beta f\|^2_s .  $$
Summing over all $|\alpha|\le p$ and setting $s=2p-n$, \eqref{YNB}~follows.

It~remains to deal with the case $n=3$. Here we can use part (g) of the last theorem: that~is,
we~need to show that
$$ \sum_{j_1,\dots,j_p=1}^{n(n-1)} \|\XX_{j_1}\dots\XX_{j_p}(f\circ\phi_a)\|^2_\pbn  $$
is bounded by a constant multiple of the same sum for~$f$ in the place of $f\circ\phi_a$.
Observe that the tangential vector-fields $\XX_j$, $j=1,\dots,n(n-1)$, span (very redundantly)
the~entire tangent space to~$\pbn$. Thus for any differentiable function $g$ on~$\pbn$ and any $x\in\pbn$,
$\sum_j\|\XX_j g(x)\|^2\asymp\|\nbt g(x)\|^2$, the~norm-square of the restriction $\nbt g(x)$ of
the tensor $\nabla g(x)$ at $x$ to the tangent space of~$\pbn$ in the sense of Riemannian geometry.
Now~for any vector field $X$ on~$\pbn$, one~has $X(g\circ\phi_a)=d\phi_a(X)g$.
Since $\phi_a$ maps the sphere $\pbn$ onto itself, the~derived map $d\phi_a$
maps the tangent space of $\pbn$ into itself. Finally, $d\phi_a|\pbn$ is a smoothly varying
map on the compact manifold~$\pbn$ (hence, in~particular, so~is its Jacobian).
Consequently,
$$ \|\nbt(g\circ\phi_a)\|^2_\pbn = \|d\phi_a(\nbt) g\|^2_\pbn \le C_a \|\nbt g\|^2_\pbn  $$
with some finite $C_a$ (independent of~$g$). Iterating this argument, it~transpires that
$$ \|\nbt^p (g\circ\phi)\|^2_\pbn \le C_a^p \|\nbt g\|^2_\pbn.  $$
Passing from $\nbt$ back to the~$\XX_j$, the~last inequality reads
$$ \sum_{j_1,j_2,\dots,j_p=1}^{n(n-1)} \|\XX_{j_1}\XX_{j_2}\dots\XX_{j_p}(g\circ\phi_a)\|^2_\pbn
 \le C_a^p \sum_{j_1,j_2,\dots,j_p=1}^{n(n-1)} \|\XX_{j_1}\XX_{j_2}\dots\XX_{j_p}g\|^2_\pbn,  $$
which is what we needed to prove.  \end{proof}

% The~last two proofs were modeled on the proofs of Theorem~19 and Corollary~10 in~\cite{EY2}.

Introduce the space
\begin{multline*}
 \HH' := \{f\text{ \Hh on }\bn: (s+n)\|f\|_s^2 \text{ has an analytic continuation}\\
 \text{ in $s$ to a neighborhood of }s=-n\} .
\end{multline*}
By~Proposition~\ref{PA}, $\HH'$~contains the algebraic span of~$\bhm$, $m\in\NN$, and the square root
$\|f\|'$ of the value of the analytic continuation above at $s=-n$ is a (semi-) norm on~$\HH'$,
with the corresponding (semi-)inner product
$$ \spr{f,g}' = \text{the analytic continuation to $s=-n$ of } (s+n)\spr{f,g}_s    $$
coinciding with $\spr{f,g}_\circ$ on~this span. It~follows that $\HH^\circ$ is just the completion
of $\HH'$ with respect to this inner product.

The~following result shows that it may be appropriate to view $\HH_\circ$ as the ``\Hh Dirichlet space'',
and gives an answer to a question on p.~180 in~\cite{St}.

\begin{theorem} \label{PF}
The inner product in $\HH_\circ$ is Moebius invariant:
$$ \spr{f\circ\phi,g\circ\phi}_\circ = \spr{f,g}_\circ  $$
for any $f,g\in\HH_\circ$ and $\phi$ in the Moebius group of~$\bn$.
\end{theorem}

\begin{proof} Since both $\HH_\circ$ and its inner product $\spr{\cdot,\cdot}_\circ$ are $\On$-invariant
(by~their very construction), it~is enough to prove the assertion for $\phi=\phi_a$;
we~can even assume that $a$ is of the form $(a,0,\dots,0)\in\bn$ with some (abusing the notation) $0\le a<1$.
Furthermore, since we know from the last proposition that the composition operator $f\mapsto f\circ\phi_a$
is continuous on~$\HH_\circ$, it~is further enough to prove the assertion for $f,g$ in a dense subset of~$\HH_\circ$.
In~particular, by~linearity, we~may assume that $f\in\bhm$ and $g\in\mathbf H^{m'}$ for some $m,m'\in\NN$.

We~will show that under all these hypotheses, $\spr{f\circ\phi_a,g\circ\phi_a}'$ exists
for all $0\le a<1$ and does not depend on~$a$. By~the observations in the paragraph
before the theorem, this will complete the proof.

Fix $0<\rho<1$. Recall that the measure
$$ d\tau(x) := \frac{dx}{(1-|x|^2)^n}  $$
on~$\bn$ is invariant under~$\phi_a$, and also
\[ 1-|\phi_a(x)|^2 = \frac{(1-|a|^2)(1-|x|^2)}{[x,a]^2}.  \label{YNC}  \]
By~the change of variable $x\mapsto\phi_a(x)$, we~thus have, for any $s>-1$,
\begin{align*}
\spr{f\circ\phi_a,g\circ\phi_a}_s
&= \frac{\Gamma(\frac n2+s+1)}{\pi^{n/2}\Gamma(s+1)} \int_\bn (f\overline g)(\phi_a(x)) (1-|x|^2)^{s+n}\,d\tau(x) \\
&= \frac{\Gamma(\frac n2+s+1)}{\pi^{n/2}\Gamma(s+1)} \int_\bn (f\overline g)(x) (1-|\phi_a(x)|^2)^{s+n}\,d\tau(x) \\
&= \int_\bn (f\overline g)(x) \Big(\frac{1-a^2}{1-2ax_1+|x|^2a^2}\Big)^{n+s} \,d\rho_s(x).
\end{align*}
Passing to the polar coordinate $x=r\zeta$, with $0\le r<1$ and $\zeta\in\pbn$, we~can continue with
\[ = \frac{\Gamma(\frac n2+s+1)}{\pi^{n/2}\Gamma(s+1)} \int_0^1 \frac{2\pi^{n/2}}{\Gamma(\frac n2)} \int_\pbn (f\overline g)(r\zeta)
 \Big(\frac{1-a^2}{1-2ar\zeta_1+r^2a^2}\Big)^{n+s} (1-r^2)^s r^{n-1} \,d\zeta \,dr, \label{XE} \]
that~is, using~\eqref{YC},
\begin{align*}
&= \frac{\Gamma(\frac n2+s+1)}{\pi^{n/2}\Gamma(s+1)} \int_0^1 G(a,r) (1-r^2)^s r^{n-1} \,dr , \\
& \hskip4em \text{where } G(a,r) := \frac{2\pi^{n/2}}{\Gamma(\frac n2)} r^{m+m'} S_m(r^2) S_{m'}(r^2)
 \int_\pbn (f\overline g)(\zeta) \Big(\frac{1-a^2}{1-2ar\zeta_1+r^2a^2}\Big)^{n+s} \,d\zeta .
\end{align*}
Carrying out the $\zeta$ integration shows that $G(a,r)$ is a holomorphic function of $|a|<\rho$
and $|r|<1/\rho$.

Recall now that if $F(t)=\sum_{j=0}^\infty F_j (1-t)^j$ is holomorphic in some neighborhood
of $t=1$ and continuous on $[0,1]$, then
\[ \mathcal I (s) := \int_0^1 F(t) (1-t)^s \,dt, \qquad s>-1,  \label{YA}  \]
extends to a holomorphic function of $s$ on the entire complex plane~$\CC$, except for simple poles
at $s=-j-1$, $j=0,1,2,\dots$, with residues~$F_j$. Differentiating \eqref{YA} under integral sign,
it~transpires also that, for any $k=0,1,2,\dots$,
$$ \mathcal I_k (s) := \int_0^1 F(t) \Big(\log\frac1{1-t}\Big)^k (1-t)^s \,dt, \qquad s>-1, $$
extends to a holomorphic function of $s$ on the entire complex plane~$\CC$, except for poles of
multiplicity $k+1$ at $s=-j-1$, $j=0,1,2,\dots$, with principal part $k!F_j/(s+j+1)^{k+1}$.
See Lemma~2 in \cite{EY2} for the details.

Now by \cite[\S2.10~(12)]{BE}, for $n>2$ odd $S_m$~is of the form $F_1(t)+F_2(t)(1-t)^{n-1}\log\frac1{1-t}$,
with $F_1,F_2$ as~in~\eqref{YA} (and~$S_m$ is actually a polynomial in $1-t$ for even~$n$).
It~thus follows from the observation in the preceding paragraph that $\spr{f\circ\phi_a,g\circ\phi_a}_s$
extends to a holomorphic function of $|a|<\rho$ and $s\in\CC$, except for at most simple poles at
$s=-n,-n-1,\dots,-2n+2$ and at most double poles at $s=-2n-j+1$, $j\in\NN$.
Consequently, the~function $(s+n)\spr{f\circ\phi_a,g\circ\phi_a}_s$ extends to a holomorphic
function of $|a|<\rho$ and $s\in\CC$ except for poles as above, excluding $s=-n$ where it
assumes a finite value. In~particular (taking $f=g$), this means that $f\circ\phi_a,g\circ\phi_a\in\HH'$
for all $0\le a<\rho$, and the inner product $\spr{f\circ\phi_a,g\circ\phi_a}'$ is a smooth
function of these~$a$.

Finally, it~is legitimate to differentiate under the integral sign in~\eqref{XE}, yielding,
for $s>-1$,
\[ \begin{aligned}
& \frac\partial{\partial a} \spr{f\circ\phi_a,g\circ\phi_a}_s = \frac{\Gamma(\frac n2+s+1)}{\pi^{n/2}\Gamma(s+1)}
 \int_0^1 \frac{2\pi^{n/2}}{\Gamma(\frac n2)} \int_\pbn (f\overline g)(\zeta) r^{m+m'} S_m(r^2) S_{m'}(r^2) \;\times \\
& \hskip4em (n+s) \Big(\frac{1-a^2}{1-2ar\zeta_1+r^2a^2}\Big)^{n+s-1} \Big[\frac\partial{\partial a}
 \frac{1-a^2}{1-2ar\zeta_1+r^2a^2} \Big] (1-r^2)^s r^{n-1} \,d\zeta \,dr . \end{aligned}  \label{XF}  \]
Repeating the argument above, it~transpires that for all $0\le a<\rho$,
\[ \frac\partial{\partial a} (n+s) \spr{f\circ\phi_a,g\circ\phi_a}_s = (n+s) F_a(s),  \label{XG} \]
where $F_a(s)$ is a holomorphic function of $s$ except for at most double poles at $s=-2n+1-j$, $j\in\NN$,
and at most simple poles at $s=-n-1,\dots,-2n+2$; in~particular, $F_a(s)$~is holomorphic near $s=-n$
and assumes a finite value there. Hence, thanks to the factor $n+s$ in~\eqref{XG},
$$ \frac\partial{\partial a} \spr{f\circ\phi_a,g\circ\phi_a}' =0 \quad\text{for }0\le a<\rho.  $$
Since $\rho$ was arbitrary, it~follows that $\spr{f\circ\phi_a,g\circ\phi_a}'=
\spr{f\circ\phi_0,g\circ\phi_0}'=\spr{f,g}'$ for all $0\le a<1$, completing the proof.  \end{proof}

% Again, t
The~last proof is modeled on the proof of Theorem~15 in~\cite{EY2}.

\section{Reproducing kernel} \label{sec6}
The~reproducing kernel of $\HH_\circ$ --- more precisely, for its subspace $\HH_{\circ,0}$ of functions
vanishing at the origin --- is~given~by
\[ K_\circ(x,y) = \sum_{m>0} \frac{S_m(|x|^2) S_m(|y|^2)}{I_m^\circ} Z_m(x,y) .  \label{KK} \]
In~this section, we~give a closed expression for this sum.

\begin{proposition} The kernel $K_\circ(x,y)$ obeys the following transformation formula:
\[ K_\circ(x,y) = K_\circ(\phi_a(x),\phi_a(y)) - K_\circ(\phi_a(x),a) - K_\circ(a,\phi_a(y)) + K_\circ(a,a) . \label{KB} \]
\end{proposition}

\begin{proof} Let $\{e_j\}$ be an orthonormal basis for $\HH_{\circ,0}$; thus by the standard formula for reproducing kernel~\cite{Aro}
\[ K_\circ(x,y) = \sum_j e_j(x) \overline{e_j(y)}.  \label{KA} \]
By~Theorem~\ref{PF}, $e_j\circ\phi_a$ will also be mutually orthogonal unit vectors in~$\HH_\circ$, for any $a\in\bn$;
since constants are orthogonal to everything in~$\HH_\circ$, it~follows that $\{e_j\circ\phi_a-e_j(a)\}$
is also an orthonormal basis for~$\HH_{\circ,0}$, whence again by~\cite{Aro}
$$ K_\circ(x,y) = \sum_j (e_j\circ\phi_a(x)-e_j(a)) \overline{(e_j\circ\phi_a(y)-e_j(a))}.  $$
Expanding the right-hand side and using~\eqref{KA}, we~get~\eqref{KB}.
\end{proof}

Introduce the function
\[ \label{fdef} f_n(x) := \int_0^x (1-t)^{n-2} \FF21{n,\frac n2}{1+\frac n2}t \,dt. \]

\begin{proposition} $K_\circ(x,x)=c_n f_n(|x|^2)$ for some constant~$c_n$ depending only on~$n$.
\end{proposition}

\begin{proof} Writing for brevity $K_\circ(x,x)=:K(x)$, setting $y=x$ in \eqref{KB} yields
$$ K(x) = K(\phi_a(x)) - K_\circ(\phi_a(x),a) - K_\circ(a,\phi_a(x)) + K(a) . $$
The~last three terms are \Hh functions; applying $\dh$ to both sides, we~thus obtain
$$ \dh K = \dh(K\circ\phi_a) = (\dh K)\circ\phi_a  $$
by~the Moebius invariance of~$\dh$. Since $a\in\bn$ is arbitrary, it~follows that
\[ \dh K\equiv c_n   \label{KC}  \]
for some constant $c_n$ depending only on~$n$.
Now~$K$ is evidently a radial function, $K(x)=:f(|x|^2)$ for some function $f$ on $(0,1)$.
Since, by~a~short computation,
$$ \dh f(|x|^2) = (1-|x|^2) [(1-|x|^2)(2nf'(|x|^2)+4|x|^2f''(|x|^2)) + 2(n-2)2|x|^2f'(|x|^2)] , $$
we~see that $f(t)$ must satisfy the differential equation
$$ (1-t) [(1-t)(2nf'+4tf'')+4(n-2)tf'] \equiv c_n,  $$
that~is,
\[ \frac{(1-t)^n}{t^{n/2-1}} \Big( \frac{t^{n/2}}{(1-t)^{n-2}} f'(t) \Big)' \equiv c_n, \label{KD} \]
with the initial condition $f(0)=0$ (since $K_\circ(0,0)=0$). This gives
$$ f(t) = c_n \int_0^t (1-t)^{n-2} \FF21{\tfrac n2,n}{1+\tfrac n2}t \, dt,  $$
i.e. $f(t)=c_n f_n(t)$.
\end{proof}

\begin{proposition}
We~have
\[ f_n(x) = \frac{n-2}{2(n-1)}x \FF32{2-\frac n2,1,1}{2,1+\frac n2}x + \frac n{2(n-1)} \log \frac1{1-x} . \label{EE} \]
% \begin{align}
% \label{E1} f(x)&= \sum_{k=0}^\infty x^{k+1} \FF32{-k,n,1}{2,\frac n2+1}1 \\
% \label{E3} &= \frac{n-2}{2(n-1)}x \FF32{2-\frac n2,1,1}{2,1+\frac n2}x + \frac n{2(n-1)} \log \frac1{1-x} \\
% \label{E4} &= \frac{n}{2(1-n)}[\epsilon]\Big((1-x)^{\epsilon} \FF21{\epsilon,1-\frac n2}{\frac n2}x\Big) \\
% \label{E5} &= \frac{n}{2(1-n)}[\epsilon] \FF21{\epsilon,n-1}{\frac n2}{\frac{x}{x-1}} \\
% \label{E6} &= \frac{n}{2(1-n)} (1-x)^{n-1}[\epsilon] \FF21{\frac n2-\epsilon,n-1}{\frac n2}x ,
% \end{align}
% where $[\epsilon]f(\epsilon):= f'(0)$.
\end{proposition}

\begin{proof} Using the Euler transform we can express the integrand in \eqref{fdef}~as
$$ (1-t)^{n-2} \FF21{n,\frac n2}{1+\frac n2}t = (1-t)^{-1} \FF21{1-\frac n2,1}{1+\frac n2}t . $$
Next we are going to use the well known contiguous relation (see~\cite[15.3.13]{NIST})
$$ (c-a-b) \FF21{a,b}{c}t + a(1-t) \FF21{a+1,b}ct - (c-b) \FF21{a,b-1}ct = 0.  $$
For $a=1-n/2$, $b=1$, $c=1+n/2$ this translates into
$$ (n-1) \FF21{1-\frac n2,1}{1+\frac n2}t = -(1-\frac n2)(1-t) \FF21{2-\frac n2,1}{1+\frac n2}t + \frac n2 \FF21{1-\frac n2,0}{1+\frac n2}t .  $$
Consequently,
$$ (1-t)^{-1} \FF21{1-\frac n2,1}{1+\frac n2}t = -\frac{2-n}{2(n-1)} \FF21{2-\frac n2,1}{1+\frac n2}t + \frac{n}{2(n-1)}\frac1{1-t}  $$
and
$$ f(x) = \int_0^x (1-t)^{-1} \FF21{1-\frac n2,1}{1+\frac n2}t \,dt = \frac{(n-2)x}{2(n-1)} \FF32{1,2-\frac n2,1}{2,1+\frac n2}x + \frac{n}{2(n-1)} \log\frac1{1-x},  $$
proving~\eqref{EE}.
%
% The identity \eqref{E4} follows from \eqref{E3} exploiting the fact that
% \[ \label{Taylrem} \FF21{a,b}{c}x = 1+\frac{ab}{c}x \;\FF32{a+1,b+1,1}{c+1,2}x.  \]
% In our case
% $$ \FF21{\epsilon,1-\frac n2}{\frac n2}x = 1+\frac{\epsilon(2-n)}{n}x \FF32{\epsilon+1,2-\frac n2,1}{1+\frac n2,2}x , $$
% and thus
% $$ [\epsilon] \FF21{\epsilon,1-\frac n2}{\frac n2}x = \frac{(2-n)x}{n} \FF32{1,2-\frac n2,1}{1+\frac n2,2}x.  $$
% Clearly,
% $$ [\epsilon] (1-x)^{\epsilon} \FF21{\epsilon,1-\frac n2}{\frac n2}x = \log(1-x) + \frac{(2-n)x}{n} \FF32{1,2-\frac n2,1}{1+\frac n2,2}x.  $$
% Finally,
% $$ \frac{n}{2(1-n)}[\epsilon](1-x)^{\epsilon} \FF21{\epsilon,1-\frac n2}{\frac n2}x = \frac{n}{2(n-1)} \log\frac1{1-x}
%  + \frac{(n-2)x}{2(n-1)} \FF32{1,2-\frac n2,1}{1+\frac n2,2}x,  $$
% proving (\ref{E4}).
%
% The identities \eqref{E5}, \eqref{E6} are obtained from \eqref{E4} by using the Pfaff and the Euler transform, respectively~\cite[\S2.1~(22)]{BE}.
%
% The last remaining identity \eqref{E1} can again be derived using~\eqref{Taylrem}:
% $$ \FF21{-k-1,n-1}{\frac n2}1 = 1+\frac{2(k+1)(n-1)}{n} \FF32{-k,n,1}{2,1+\frac n2}1 ,  $$
% and applying the Gauss summation formula on $\!_2F_1(1)$. This yields
% $$ \FF32{-k ,n,1}{2,1+\frac n2} = \frac{\frac{(1-\frac n2)_{k+1}}{(\frac n2)_{k+1}}-1}{\frac{2(k+1)(n-1)}{n}}.  $$
% This is enough to establish the identity \eqref{E1}=\eqref{E3}.
\end{proof}

\begin{theorem} The~reproducing kernel is given~by
\[ K(x,y) = \frac{2(n-1)}n [ f_n(|x|^2) + f_n(|y|^2) - f_n(|\phi_y(x)|^2) ],  \label{KI} \]
with $f_n$ as in~\eqref{EE}.
\end{theorem}

\begin{proof} Replacing $x$, $y$ in \eqref{KB} by $\phi_a(x)$, $\phi_a(y)$, respectively,
and setting $y=x$, we~get (recall that $\phi_a\circ\phi_a=\text{id}$)
$$ K_\circ(\phi_a(x),\phi_a(x)) = K_\circ(x,x) - K_\circ(x,a) - K_\circ(a,x) + K_\circ(a,a) ,  $$
that~is, replacing $a$ by~$y$,
$$ K_\circ(x,y) = \frac{K_\circ(x,x)+K_\circ(y,y)-K_\circ(\phi_y(x),\phi_y(x))}2  $$
(recall that $K_\circ(x,y)$ is real-valued and $K_\circ(y,x)=K_\circ(x,y)$ --- a~reflection of the fact that
$H$-harmonicity is preserved by complex conjugation). By~the last two propositions, this~gives
$$ K_\circ(x,y) = \frac{c_n}2 [ f_n(|x|^2) + f_n(|y|^2) - f_n(|\phi_y(x)|^2) ] . $$
It~remains to compute the constant~$c_n$. To~this end, we~momentarily return to the function
$K_\circ(x,x)=K(x)=c_nf_n(|x|^2)$ above, and compare the behavior at the origin~of
\[ K_\circ(x,x) = c_n f_n(|x|)^2 \label{EA} \]
and
\[ K_\circ(x,x) = \sum_{m>0} \frac{S_m(|x|^2)^2}{I_m^\circ} Z_m(x,x) \label{EB}  \]
(cf.~\eqref{KK}). Both \eqref{EA} and~\eqref{EB} are functions of $t:=|x|^2$, and the derivative
of~\eqref{EA} with respect to $t$ at $t=0$ equals $c_n$ (since $f'_n(0)=1)$.
For~the analogous derivative of \eqref{EB} at $t=0$, the only contribution comes from the term $m=1$
(the~terms with $m>1$ have double or higher order zeros at $t=0$), and is easily seen to be
equal to $4(n-1)/n$. Thus $c_n=4(n-1)/n$, completing the proof.
\end{proof}

It~seems quite difficult to obtain \eqref{KI} by directly summing~\eqref{KK}.

\section{Concluding remarks} \label{sec7}
%\begin{remark}
In~the holomorphic, pluriharmonic and harmonic cases, the~analogues of the $I_m(s)$
are actually continuous and positive for all $s>-n-1$ (``the~Wallach set''), hence one takes just
the limit instead of the analytic continuation. The~\Hh case seems to be fundamentally different:
as~$s$ decreases, $I_m(s)$~become $1$ when $s=-1$ (this is no surprise --- the Hardy space situation),
then seem again to become all $1$ at $s=-2$ (this is already quite unexpected and definitely has no
parallel in the above three cases), and a little thereafter become negative (!) in general,
only to ``re-surface'' into our Dirichlet space at $s=-n$. Furthermore --- this is again highly
surprising and without any analogue in the three cases above --- our Dirichlet space seems to
appear once again at $s=-n-1$: namely, the residue of $I_m(s)$ at $s=-n-1$ is the same as at $s=-n$,
for all $m\in\NN$. Understanding the Wallach set,
i.e.~the set of all $s\in\CC$ for which $1/I_m(s)\ge0$ $\forall m$, would be highly desirable.

More concretely, by~analogy with the situation for spaces of holomorphic functions on bounded
symmetric domains in~$\CC^n$ (see~e.g.~\cite{Ara}), let us define the \emph{continuous \Hh Wallach set}~by
$$ \WW_c :=\{s\in\CC: 0<I_m(s)<+\infty \;\forall m\in\NN\} ,   $$
and the \emph{discrete \Hh Wallach set of level~$k$}, $k=1,2,\dots,$~by
\begin{multline*}
 \WW_d^k :=\{s\in\CC: I_m \text{ has a pole at $s$, and the $-k$-th Laurent coefficients } \\
 \text{are of the same sign, } \forall m\in\NN\}.
\end{multline*}
(Thus $-n,-n-1\in\WW_d^1$.) For~each $s\in\WW_c$, we~have the associated reproducing kernel space~$\HH_s$
(an ``analytic continuation'' of our weighted \Hh Bergman spaces); and for each $s\in\WW_d^k$,
one~can construct an analogue of our Dirichlet space (see \cite{EY2} for the $M$-harmonic situation with $k=2$).

\medskip

{\bf Question 1.} What is $\WW_c$, and what are $\WW_d^k$, $k=1,2,\dots$?

\medskip

{\bf Question 2.} Are there any more cases where $I_m(s_1)=I_m(s_2)$ $\forall m$ with $s_1<s_2$,
like the above situation $s_1=-1$, $s_2=-2$? Are there any more cases when $I_m$ would have poles
of the same order and strength for all~$m$ at two different values of~$s$, like the above situation
at $s=-n$ and $s=-n-1$?

\medskip

Note also that~\eqref{YF}, while valid for all $s\ge-1$, fails for $s=-2$ (since $I_m(-2)=1$ for all~$m$).

\medskip

{\bf Question 3.} What is the behavior of $I_m(s)$ as $m\to\infty$, for general fixed $s\in\CC$?

\medskip

Quite generally, for any real $s$ we can introduce the spaces
$$ \HH_{\#s} := \{f=\sum_m f_m\text{ \Hh on }\bn: \;\|f\|^2_{\#s}:=\sum_m (m+1)^{-s-1}\|f_m\|^2_\pbn<+\infty \} ,  $$
(where $f_m\in\bhm$ are, as~before, the Peter-Weyl components of~$f$). Clearly $\HH_{\#s}$ are Banach spaces,
and $\HH_{\#s_1}\subset\HH_{\#s_2}$ continuously for $s_1<s_2$. Also by~\eqref{YF},
$$ \HH_{\#s} = \HH_s \qquad\text{for } s>-1,  $$ %  \label{HA}
with equivalent norms; while by~\eqref{YH},
$$ \HH_{\#-n} = \HH_\circ, $$
our \Hh Dirichlet space, with the norm $\|f\|_{\#-n}$ equivalent to $\|f\|_\circ+|f(0)|$.
On~the other hand, the~observations in the preceding paragraphs show that $\HH_{-2}$ does not
coincide with~$\HH_{\#-2}$, but instead $\HH_{-2}=\HH_{\#-1}$ is again the \Hh Hardy space of Stoll~\cite{St2}.
Understanding the spaces $\HH_{\#s}$ for general $s\in\RR$ is tantamount to having the answer to Question~3 above.
%\qed \end{remark}

Our~last remark concerns the semi-inner product
\[ \spr{f,g}_\circ = \sum_m I_m^\circ \spr{f,g}_\pbn     \label{EC}  \]
in our Dirichlet space~$\HH_\circ$.

\medskip

{\bf Question 4.} Is~\eqref{EC} the unique (up~to constant multiple) Moebius-invariant
inner product on \Hh functions on~$\bn$?

\medskip

For~the holomorphic case, the corresponding assertion is true \cite[Theorem~6.16]{Zhu};
in~the harmonic case it~does not make sense, and for the $M$-harmonic case it is wide~open.

\end{document}